\documentclass[11pt]{article}
\usepackage[margin =1in]{geometry}
\usepackage{amssymb,amsthm,amsmath}
\usepackage{enumerate}
\usepackage{hyperref}
\usepackage{cite}
\usepackage[capitalize]{cleveref}
\usepackage{enumitem}
\usepackage{color}
\usepackage{cancel}
\usepackage{pgf}
\usepackage{svg}
\usepackage{pgfplots}
\usepackage{import}

\usepackage[utf8]{inputenc} 
\usepackage[T1]{fontenc}    
\usepackage{hyperref}       
\usepackage{url}            
\usepackage{booktabs}       
\usepackage{amsfonts}       
\usepackage{nicefrac}       
\usepackage{microtype}
\usepackage{multirow}
\usepackage{xfrac}
\usepackage{todonotes}
\usepackage{titlesec}
\usepackage{caption}
\usepackage{subcaption}
\usepackage{algorithm}
\usepackage{algorithmic}
\definecolor{blue}{gray}{0.0}
\titleformat{\subsubsection}[runin]
{\normalfont\normalsize\bfseries}{\thesubsubsection}{1em}{}

\usepackage{graphicx}

\usepackage{appendix}
\numberwithin{equation}{section}

\newcommand{\inclu}[0] {\ar@{^{(}->}}

\newcommand{\argmin}{\operatornamewithlimits{argmin}}

\newtheorem{thm}{Theorem}[section]
\newtheorem{theorem}{Theorem}[section]

\newtheorem{lemma}[thm]{Lemma}

\newtheorem{conjecture}[thm]{Conjecture}

\crefname{claim}{claim}{claims}
\Crefname{claim}{Claim}{Claims}
\crefname{lem}{lemma}{lemmas}
\Crefname{lem}{Lemma}{Lemmas}
\crefname{algorithm}{algorithm}{algorithms}
\Crefname{algorithm}{Algorithm}{Algorithms}

\theoremstyle{remark}

\usepackage{mathtools}

\usepackage[algo2e, boxruled]{algorithm2e}

\usepackage[T1]{fontenc}
\usepackage{lmodern}

\begin{document}

	\title{On Optimal Universal First-Order Methods\\ for Minimizing Heterogeneous Sums}
	\author{Benjamin Grimmer\footnote{\texttt{grimmer@jhu.edu}, Johns Hopkins University, Department of Applied Mathematics and Statistics}}
	\date{}
	\maketitle

	\begin{abstract}
		This work considers minimizing a sum of convex functions, each with potentially different structure ranging from nonsmooth to smooth, Lipschitz to non-Lipschitz. Nesterov's universal fast gradient method~\cite{Nesterov2015UniversalGM} provides an optimal black-box first-order method for minimizing a single function that takes advantage of any continuity structure present without requiring prior knowledge. In this paper, we show that this landmark method (without modification) further adapts to heterogeneous sums. For example, it minimizes the sum of a nonsmooth $M$-Lipschitz function and an $L$-smooth function at a rate of $ O(M^2/\epsilon^2 + \sqrt{L/\epsilon}) $ without knowledge of $M$, $L$, or even that the objective was a sum of two terms. This rate is precisely the sum of the optimal convergence rates for each term's individual complexity class. More generally, we show that sums of varied H\"older smooth functions introduce no new complexities and require at most as many iterations as is needed for minimizing each summand separately. Extensions to strongly convex and H\"older growth settings as well as simple matching lower bounds are also provided.
	\end{abstract}
    
    {\small {\bf Keywords:} finite sums, holder continuity, optimal, universal gradient method, first-order method
    
    }
    
    \section{Introduction}
In this paper, we are interested in first-order methods for approximately solving convex optimization problems of the form
\begin{equation}\label{eq:our-problem}
    p_* = \min_{x\in Q} F(x) + \Psi(x)
\end{equation}
where $F(x) = \sum_{j\in\mathcal{J}} f_j(x)$ is given by a sum of functions $f_j$ each individually possessing some standard structure (Lipschitz continuity, smoothness, or more generally H\"older smoothness). This sum is heterogeneous in that it is a composition of several terms ranging from smooth to nonsmooth, Lipschitz to non-Lipschitz. Note typically, $F$ will not possess any of the structure held by its components. {\color{blue} The constraint set $Q$ and additive term $\Psi$ are assumed to be closed, convex and simple.}

{\color{blue} We consider methods given a first-order oracle\footnote{{\color{blue} In particular, the considered methods require an oracle producing the function value and one (sub)gradient of $F$ at the current iterate, to be called at each iteration.}}} for $F$, seeking an approximate $\epsilon>0$-minimizer {\color{blue} $y$} satisfying $F(y) + \Psi(y) - p_* \leq \epsilon$. In the landmark paper~\cite{Nesterov2015UniversalGM}, Nesterov introduced the Universal Fast Gradient Method ($\mathtt{UFGM}$), which we review in Section~\ref{sec:prelim}. This method was analyzed for minimizing $F(x)+\Psi(x)$ where $F=f$ is a single $(M,v)$-H\"older smooth function, defined for $M\geq 0$ and $v\in[0,1]$ as
\begin{equation} \label{eq:holder-smooth}
	\|\nabla f(x) - \nabla f(y)\|_* \leq M\|x-y\|^{v}, \quad x,y\in Q \ . 
\end{equation}
Note $(L,1)$-H\"older smoothness corresponds to the typical smooth optimization assumption of an $L$-Lipschitz gradient, and $(M,0)$ corresponds to the typical nonsmooth optimization assumption of having an $M$-Lipschitz objective function.
$\mathtt{UFGM}$'s iterates $y_k$ are all $\epsilon$-minimizers once
\begin{equation} \label{eq:Nesterov-rate}
	k \geq 2^{\frac{3+5v}{1+3v}}\left(\frac{M}{\epsilon}\right)^{\frac{2}{1+3v}}\xi(x_0,x^*)^{\frac{1+v}{1+3v}} \ ,
\end{equation}
where $x^*\in\argmin_{x\in Q} F(x)+\Psi(x)$ and $\xi(x_0,x^*)$ is a Bregman divergence, formally defined in~\eqref{eq:Bregman}, measuring the initial distance from optimality.
This rate cannot be improved upon with respect to all three of $M,\epsilon,\xi(x_0,x^*)$ for any H\"older exponent $v\in [0,1]$~\cite{Nemirovskii1986}. {\color{blue} Hence we say this algorithm is ``optimal''.} Moreover, $\mathtt{UFGM}$ does not require knowledge of any of these parameters{\color{blue}, applying equally across the range of nonsmooth to smooth problems. Such widely applicable methods are called ``universal''.}

In this paper, we extend this optimal, universal theory to apply to generic sums of H\"older smooth functions of the form~\eqref{eq:our-problem}. {\color{blue} We assume that the sum $F = \sum_{j\in\mathcal{J}} f_j$ is convex, that the sum rule $\sum_{j\in\mathcal{J}} \nabla f_j = \nabla F$ holds, and that each $\nabla f_j$ satisfies $(M_j,v_j)$-H\"older smoothness. Note this allows individual $f_j$ terms to be nonconvex, as the sum rule still holds for all $C^1$ functions.}
As a simple first example, consider $F(x)=\frac{1}{2}(|x|+x^2)$, which is not H\"older smooth for any $(M,v)$ despite being the sum of $(1,0)$ and $(1,1)$-H\"older smooth functions. {\color{blue} Section~\ref{sec:applications} provides more applied examples: mixtures of maximum likelihood models, vector machine training, and projection onto spectrahedrons.}

Our main result finds that the performance of $\mathtt{UFGM}$ on such a sum is simply given by summing up all the individual convergence rates~\eqref{eq:Nesterov-rate}. Hence the method's behavior can be understood as the superposition of each summand's individual H\"older smooth setting. This is stated below with proofs deferred to Section~\ref{sec:theory}.
\begin{theorem}\label{thm:main-rate}
	For any convex $F(x) = \sum_{j\in \mathcal{J}}f_j(x)$ where each $f_j$ is $(M_j,v_j)$-H\"older smooth~\eqref{eq:holder-smooth} and target accuracy $\epsilon>0$, the iterates $y_k$ of the $\mathtt{UFGM}$ are $\epsilon$-minimizers of~\eqref{eq:our-problem} for all
	\begin{equation}\label{eq:our-rate}
	    k \geq \sum_{j\in\mathcal{J}}\left[c_j\left(\frac{M_j}{\epsilon}\right)^{\frac{2}{1+3v_j}}\xi(x_0,x^*)^{\frac{1+v_j}{1+3v_j}} \right]
	\end{equation}	
	with coefficients $c_j = \frac{1+3v_j}{1+v_j}2^{\frac{2+2v_j}{1+3v_j}} |\mathcal{J}|^{\frac{1-v_j}{1+3v_j}}$ and $x^*$ minimizing~\eqref{eq:our-problem}.
\end{theorem}
Importantly, we do not modify Nesterov's method at all, only providing it a {\color{blue} first-order oracle} for the overall summation $F$. Consequently, we arrive at a stronger statement of $\mathtt{UFGM}$'s universality, it adapts to sums of structured functions without knowledge of their types or even the number of summands. 

Many optimization problems take the form of minimizing a nonsmooth Lipschitz function $f_1$ plus a smooth function $f_2$ (see the example applications in Section~\ref{sec:applications}). For such applications where $f_1$ is $(M,0)$-H\"older smooth and $f_2$ is $(L,1)$-H\"older smooth, Theorem~\ref{thm:main-rate} guarantees $\mathtt{UFGM}$ has $y_k$ as an $\epsilon$-minimizer of $F(x) = f_1(x)+f_2(x)$ for all
\begin{equation}\label{eq:our-two-term-rate}
    k \geq 8\left(\frac{M}{\epsilon}\right)^{2}\xi(x_0,x^*) + 4\sqrt{\frac{L \xi(x_0,x^*)}{\epsilon}} \ .
\end{equation}
Up to small constants, this rate is the sum of the optimal rates for nonsmooth $M$-Lipschitz minimization and $L$-smooth minimization.
The optimality of Theorem~\ref{thm:main-rate} in the Euclidean setting (where $\xi(x,y)=\frac{1}{2}\|x-y\|^2$) follows immediately from the known lower bounds (see~\cite{Lan2015} which established similar optimal universal guarantees for a level-bundle method): At least
\begin{equation} \label{eq:lower-bound}
    \max_{j\in\mathcal{J}} \left[c_j'\left(\frac{M_j}{\epsilon}\right)^{\frac{2}{1+3v_j}}R^{\frac{1+v_j}{1+3v_j}} \right] \geq \sum_{j\in\mathcal{J}} \left[ \frac{c_j'}{|\mathcal{J}|}\left(\frac{M_j}{\epsilon}\right)^{\frac{2}{1+3v_j}}R^{\frac{1+v_j}{1+3v_j}} \right]
\end{equation}
{\color{blue} first-order oracle} evaluations are required in the worst case to find an $\epsilon$-minimizers where the coefficients $c_j'>0$ depend only on $v_j$.

We numerically observe convergence matching~\eqref{eq:our-two-term-rate} in Figure~\ref{fig:p-norm} (see Section~\ref{sec:applications}). By varying the size of $M$, we see that the universal method converges at an accelerated rate until reaching an accuracy on the order of $\epsilon = O(M^{4/3})$ and then the method's convergence slows down as the nonsmooth term dominates.

\paragraph{Outline.} In the remainder of this introduction, we discuss extensions of our main result to strongly convex problems (and more generally growth/error bounded settings), scaling with respect to $|\mathcal{J}|$, and the importance of universal, blackbox results on such heterogeneous sums. Section~\ref{sec:prelim} briefly introduces Nesterov's universal fast gradient method. Then in Section~\ref{sec:applications}, we discuss applications and simple numerics showing a transition from fast smooth convergence to slow nonsmooth convergence as the dominant term in~\eqref{eq:our-two-term-rate} changes. Finally, Section~\ref{sec:theory} proves our main theorems. 

\subsection{Improved Convergence Guarantees Under H\"older Growth Bounds} \label{subsec:improved-restart}
Many works~\cite{Bolte2017,Yang2018,Roulet2020,RenegarGrimmer2018} have shown improved convergence guarantees whenever a growth bound
\begin{equation}\label{eq:holder-growth}
    F(x) +\Psi(x) - p_* \geq \mu\xi(x,x^*)^{p/2}
\end{equation}
holds, which we refer to as $(\mu,p)$-H\"older growth. These conditions are closely related to the Kurdyka-{\L}ojasiewicz condition~\cite{Kurdyka1998}, which are widespread, holding for generic subanalytic functions~\cite{Lojasiewicz1963,Lojasiewicz1993} and nonsmooth subanalytic convex functions~\cite{Bolte2007}.

In the setting of Euclidean distances, the recent work of~\cite{Roulet2020} showed for any $(M,v)$-H\"older smooth function with $(\mu,p)$-H\"older growth, a restarted variant of $\mathtt{UFGM}$ finds an $\epsilon$-minimizer within
\begin{equation}\label{eq:optimal-growth-rate}
    \begin{cases} O\left(\frac{M^{\frac{2}{1+3v}}}{\mu^{\frac{2(1+v)}{p(1+3v)}}\epsilon^{\frac{2(p-1-v)}{p(1+3v)}}}\right) & \text{ if } v<p-1\\
    O\left(\left(\frac{M}{\mu}\right)^{\frac{2}{1+3v}}\log(1/\epsilon)\right) & \text{ if } v=p-1
    \end{cases} 
\end{equation}
iterations. Our analysis directly extends this showing a convergence rate for a sum of $(M_j,v_j)$-H\"older smooth functions, which satisfies $(\mu,p)$-H\"older growth, equal to the sum of the individual $M_j,v_j,\mu,p$ rates of~\eqref{eq:optimal-growth-rate}. As our focus is not on the details of restarting schemes, we analyze a simple restarted method ($\mathtt{R\mbox{-}UFGM}$ defined in Algorithm~\ref{alg:RUFGM}) which assumes knowledge of the optimal objective value.
\begin{theorem}\label{thm:holder-growth-rate}
    For any convex $F(x) = \sum_{j\in \mathcal{J}}f_j(x)$ satisfying $(\mu,p)$-H\"older growth~\eqref{eq:holder-growth} where each $f_j$ is $(M_j,v_j)$-H\"older smooth~\eqref{eq:holder-smooth} and target accuracy $\tilde \epsilon>0$, $\mathtt{R\mbox{-}UFGM}$ finds an $\tilde\epsilon$-minimizers of~\eqref{eq:our-problem} by iteration
	\begin{equation}\label{eq:strong-convexity-rate}
	    \sum_{j\in\mathcal{J}}\left[\left(c_j''\min\left\{\frac{2^{\frac{2(p-1-v_j)}{p(1+3v_j)}}}{2^{\frac{2(p-1-v_j)}{p(1+3v_j)}}-1}, \frac{N}{2^{\frac{2(p-1-v_j)}{p(1+3v_j)}}}\right\}\right)\frac{M_j^{\frac{2}{1+3v_j}}}{\mu^{\frac{2(1+v_j)}{p(1+3v_j)}}\tilde \epsilon^{\frac{2(p-1-v_j)}{p(1+3v_j)}}}\right]  +N
	\end{equation}	
	where $N=\lceil\log_2(F(z_0)+\Psi(z_0)-p_*)/\tilde\epsilon)\rceil$ and $c_j'' = \frac{1+3v_j}{1+v_j}2^{\frac{ (v_j-1)(p-2)}{p(1+3v_j)}}|\mathcal{J}|^{\frac{1-v_j}{1+3v_j}}$.
\end{theorem}

\subsection{Improved Convergence Guarantees with respect to $|\mathcal{J}|$} \label{subsec:improved-J}
For any fixed number of summands, Theorem~\ref{thm:main-rate} and equation~\eqref{eq:lower-bound} agree up to their constant coefficients $c_j$ and $c'_j/|\mathcal{J}|$. Hence the fast universal method optimally adapts to any fixed sum structure.
However, the dependence on the number of summands can be improved as its power does not agree between our upper and lower bounds. We conjecture the following optimal dependence.  
\begin{conjecture}\label{con:optimal}
    The optimal first-order oracle complexity for minimizing a convex sum $\sum_{j\in\mathcal{J}} f_j(x)$ of $(M_j,v_j)$-H\"older smooth functions to a target accuracy $\epsilon>0$ given $\frac{1}{2}\|x_0-x^*\|^2\leq R$ is
    $$ \sum_{j\in\mathcal{J}} \left[\bar c_j|\mathcal{J}|^{\frac{1-3v_j}{1+3v_j}}\left(\frac{M_j}{\epsilon}\right)^{\frac{2}{1+3v_j}}R^{\frac{1+v_j}{1+3v_j}} \right]$$
    where $\bar c_j$ depends only on universal constants and $v_j$.
\end{conjecture}
To motivate this conjecture and the necessity of a dependence on $|\mathcal{J}|$ consider the following setting:
Given a $(M,v)$-H\"older-smooth function $f$ to minimize, the optimal convergence rate is given by~\eqref{eq:Nesterov-rate}. For any $|\mathcal{J}|$, we can write this problem as minimizing a sum of $(M/|\mathcal{J}|,v)$-H\"older smooth functions given by $\sum_{j\in\mathcal{J}} [f(x)/|\mathcal{J}|]$. The optimal convergence rate for each summand here is on the order of $(M/|\mathcal{J}|\epsilon)^{2/(1+3v)})\xi(x_0,x^*)^{(1+v)/(1+3v)}$. For the sum these individual rates to match the optimal convergence guarantee for $f$, the coefficient $\bar c$ when summing the individual weights must depend on $|\mathcal{J}|$ as $$ \sum_{j\in\mathcal{J}}\left[\bar c\left(\frac{M}{|\mathcal{J}|\epsilon}\right)^{\frac{2}{1+3v}}\xi(x_0,x^*)^{\frac{1+v}{1+3v}} \right]= \Theta\left(\left(\frac{M}{\epsilon}\right)^{\frac{2}{1+3v}}\xi(x_0,x^*)^{\frac{1+v}{1+3v}}\right) \ , $$
amounting to $ \bar c=\Theta\left(|\mathcal{J}|^{\frac{1-3v}{1+3v}}\right)$.
As a direction toward tightening this gap in our theory, we derive the following implicitly defined convergence guarantee for $\mathtt{UFGM}$ in Section~\ref{subsec:implicit}.
\begin{theorem}\label{thm:main-rate-implicit}
	For any convex $F(x) = \sum_{j\in \mathcal{J}}f_j(x)$ where each $f_j$ is $(M_j,v_j)$-H\"older smooth~\eqref{eq:holder-smooth} and target accuracy $\epsilon>0$, the iterates $y_k$ of the $\mathtt{UFGM}$ are $\epsilon$-minimizers of~\eqref{eq:our-problem} for all $k \geq 5K$ where $K$ is the unique positive solution to
    $$\sum_{j\in \mathcal{J}} \frac{|\mathcal{J}|^{\frac{1-v_j}{1+v_j}}M_j^{\frac{2}{1+v_j}}\xi(x_0,x^*)}{\epsilon^{\frac{2}{1+v_j}}}K^{\frac{-(1+3v_j)}{1+v_j}} = 1 \ .$$
\end{theorem}
\noindent A short calculation\footnote{Namely, supposing all $(M_j,v_j) = (M/|\mathcal{J}|, v)$, the defining equation for $K$ simplifies to $$ |\mathcal{J}|\left(\frac{|\mathcal{J}|^{\frac{1-v}{1+v}}(M/|\mathcal{J}|)^{\frac{2}{1+v}}\xi(x_0,x^*)}{\epsilon^{\frac{2}{1+v}}}K^{\frac{-1-3v}{1+v}}\right) = \frac{M^{\frac{2}{1+v}}\xi(x_0,x^*)}{\epsilon^{\frac{2}{1+v}}}K^{\frac{-1-3v}{1+v}} = 1 \ ,$$ solved by $K = \Theta\left(\left(\frac{M}{\epsilon}\right)^{\frac{2}{1+3v}}\xi(x_0,x^*)^{\frac{1+v}{1+3v}}\right)$, matching the optimal rate for $(M,v)$-H\"older smooth minimization.} shows this maintains the optimal guarantee for the motivating example above as $|\mathcal{J}|$ grows (unlike our Theorem~\ref{thm:main-rate}).

\subsection{Related Works}

\paragraph{Importance of Universal, Blackbox Guarantees}
An algorithm is universal if it applies across a range of problem parameters (e.g., different levels of H\"older-smoothness or the existence of different growth conditions). The universal fast gradient method of Nesterov is one such algorithm, applying to a generic H\"older smooth objective $f+\Psi$, only needing access to function and first-order evaluations $(f(x_k),\nabla f(x_k))$ and a target accuracy $\epsilon>0$.
A few varied examples of other universal optimization methods and analysis: bundle methods~\cite{Lan2015,DiazGrimmer2021}, solving stochastic variational inequalities~\cite{Anthonakopoulos2021}, and Newton's method~\cite{Doikov2022}.

The universal level-bundle method analysis of Lan~\cite{Lan2015} applies widely to compositions $\Psi(f_1(x),\dots, f_m(x))$ with heterogeneous $f_j$. As a special case, their theory covers summations via $\Psi(v) = \sum_j v_j$ with rates matching our Theorem~\ref{thm:main-rate}. Their level bundle method requires more structure than we assume (needing compactness of $Q$ to compute lower bounds on the optimal value from) and operates using a projection subproblem rather than $\mathtt{UFGM}$'s Bregman step~\eqref{eq:Bregman-step}.

Typically universal methods are adaptive or blackbox, meaning they do not require the input of constants related to whatever problem structures exist. Adaptivity is of real practical importance, classically motivating in linesearching and trust-region methodologies~\cite{NocedalWright-textbook}. The restarting schemes~\cite{Roulet2020,RenegarGrimmer2018} adapt to whatever H\"older growth exists. Nesterov's universal methods adapt to whatever H\"older smoothness exists, learning an inexact smoothness constant $L_k$ over time. Moreover, from our analysis, it adapts to sums of H\"older smooth terms without knowledge of how $F(x)=\sum_{j\in\mathcal{J}} f_j(x)$ can be written as a sum (i.e., the number of terms and H\"older smoothness of each) is used.

As a benefit of this adaptivity, if many such formulations exist, converge occurs at least as fast as the infimum of~\eqref{eq:our-rate} or~\eqref{eq:strong-convexity-rate} over all such formulations. For example, this offers the following improvement for minimizing a $\mu$-strongly convex function $f$ over a compact domain $Q$. Let $M(f,Q)$ denote the Lipschitz constant of $f$ on $Q$. Then the classic convergence guarantee gives a rate of
$ O\left(M(f,Q)^2/\mu\epsilon \right). $
However, $f$ can be rewritten as the sum of two convex functions $(f(x)-\frac{\mu}{2}\|x\|^2) +\frac{\mu}{2}\|x\|^2$. Then Theorem~\ref{thm:holder-growth-rate} gives the following rate with potentially much smaller constants
$$ O\left(\frac{M(f-\frac{\mu}{2}\|\cdot\|^2, Q)^2}{\mu\epsilon} + \log\left(\frac{f(x_0)-\inf_{x\in Q}f(x)}{\epsilon}\right)\right)\ . $$

\paragraph{Importance of Heterogeneous Summations}
Minimizing finite sums have attracted substantial interest, typically assuming a common structure among all the summands. Variance reduction gives a tractable stochastic approach when the number of terms is large. However, such methods do not fit within our blackbox model as they rely on knowing the structure of $F$. Arjevani et al.~\cite{Arjevani2020} provide an approach to minimizing sums without using indexing information.

The recent work of Wang and Zhang~\cite{Wang2022} is motivated similarly to us. They consider minimizing heterogeneous sums of $L_i$-smooth $\mu_i$-strongly convex functions given a gradient oracle for individual terms. They show a method using variance reduction attains the optimal rate of $O\left(m+\frac{\sum_i \sqrt{L_i}}{\sqrt{\sum_i\mu_i}} \log(1/\epsilon)\right)$ individual function evaluations. This is the same rate with respect to $L_i$ and $\mu_i$ given by applying our Theorem~\ref{thm:holder-growth-rate} as such a sum has $(\sum \mu_i,2)$-H\"older growth\footnote{Note this comparison is somewhat superficial as we assume a more expensive oracle giving (sub)gradients of the whole objective. Given our oracle, considering all the terms as one $L=\sum L_i$-smooth term gives a faster $O\left(\frac{\sqrt{\sum_i L_i}}{\sqrt{\sum_i\mu_i}} \log(1/\epsilon)\right)$ rate.}.

Several past works have considered sums of smooth and nonsmooth but Lipschitz terms. The optimal rate~\eqref{eq:our-two-term-rate} in this setting were shown by~\cite{Xiao2009,Chen2012,Ghadimi2012} for several dual averaging methods. A normalized subgradient method was analyzed in~\cite{Grimmer2019} that converges at the sum of (sub)gradient descent's suboptimal convergence rates. In part, this work aims to follow up on these ideas.

Section~\ref{sec:applications} discusses two applications where heterogeneous sums naturally occur (maximum likelihood estimation over heterogeneous data sources and support vector machine training). Since our guarantees are universal and optimal {\color{blue} for general summation minimization, further improvements on these problems would require algorithms customized to the particular structure of such problems}.
    \section{Preliminaries and the Universal Fast Gradient Method}\label{sec:prelim}
Notationally, we closely follow~\cite{Nesterov2015UniversalGM} to ease the development of our analysis. We assume access to a first-order oracle
$ x \mapsto (F(x), \nabla F(x))$
where $\nabla F(x)$ is a subgradient of $F$ at $x$. (Note a subgradient rather than gradient oracle is needed here since $(M,0)$-H\"older smoothness only corresponds to Lipschitz continuity of the objective function, rather than some continuity of the gradient.) By the sum rule of subgradient calculus, this could be implemented as a sum over each summand $(F(x), \nabla F(x)) = \sum_{j\in\mathcal{J}}(f_j(x), \nabla f_j(x))$.

For any convex $d(x)$ satisfying the following strong convexity condition (with parameter one)
$ d(y) \geq d(x) + \langle \nabla d(x), y-x\rangle + \frac{1}{2}\|y-x\|^2$ for all $x,y\in \mathrm{rint }\ Q , $
we consider the associated {\it Bregman distance} (or divergence) 
\begin{equation} \label{eq:Bregman}
    \xi(x,y) = d(y) - (d(x) + \langle \nabla d(x), y-x\rangle) \ .
\end{equation}
When $d(x) = \frac{1}{2}\|x\|^2$, this recovers the Euclidean distance $\frac{1}{2}\|y-x\|^2$.

We assume that we can compute the following {\it Bregman mapping} given functions $F,\Psi$ and set $Q$ (either in closed form or by some efficient subroutine)
\begin{equation} \label{eq:Bregman-step}
    \argmin_{y\in Q}\left\{ \Psi_M(x,y) := F(x)+\langle \nabla F(x), y-x\rangle + M\xi(x,y) + \Psi(y)\right\}.
\end{equation}
This amounts to requiring that the constraints $Q$ and generic function $\Psi$ are sufficiently simple.
Based on this operation, given any target accuracy $\epsilon>0$, the Universal Fast Gradient Method {\color{blue} is as defined in} Algorithm~\ref{alg:UFGM}. 

\begin{algorithm}
    \caption{Universal Fast Gradient Method ($\mathtt{UFGM}$) of Nesterov~\cite{Nesterov2015UniversalGM}}\label{alg:UFGM}
    \begin{algorithmic}[1]
        \STATE \textbf{Initialization}: Choose $x_0\in Q$, $\epsilon>0$, $L_0>0$. Define $y_0=x_0$, $A_0=0$, $\phi_0(x)=\xi(x_0,x)$.
        \FOR{$k = 0,1,2,\dots$}
        \STATE Find $v_k = \argmin_{x\in Q} \phi_{k}(x)$.
        \STATE Find the smallest integer $i_k\geq 0$ such that the definitions
        \begin{align*}
            a_{k+1,i_k}^2     &= \frac{1}{2^{i_k}L_k}(A_k + a_{k+1,i_k}), \ a_{k+1,i_k} >0\\
            A_{k+1,i_k}       &= A_k + a_{k+1,i_k}\\
            \tau_{k,i_k}      &= \frac{a_{k+1,i_k}}{A_{k+1,i_k}}\\
            x_{k+1,i_k}       &=\tau_{k,i_k}v_k + (1-\tau_{k,i_k})y_k\\
            \hat{x}_{k+1,i_k} &= \argmin_{y\in Q}\left\{\xi(v_k,y) + a_{k+1,i_k}[\langle\nabla F(x_{k+1,i_k}),y\rangle + \Psi(y)]\right\}\\
            y_{k+1,i_k} &= \tau_{k,i_k}\hat{x}_{k+1,i_k} + (1-\tau_{k,i_k})y_k
        \end{align*}
        satisfy
        \begin{align}
             F(y_{k+1,i_k}) \leq F(x_{k+1,i_k}) &+ \langle\nabla F(x_{k+1,i_k}), y_{k+1,i_k} - x_{k+1,i_k}\rangle \nonumber\\
             &+ 2^{i_k-1}L_k\|y_{k+1,i_k} - x_{k+1,i_k}\|^2 + \frac{\epsilon\tau_{k,i_k}}{2}. \label{eq:descent}
        \end{align}
        \STATE Set $x_{k+1}=x_{k+1,i_k}$, $y_{k+1} = y_{k+1,i_k}$, $a_{k+1}=a_{k+1,i_k}$, $\tau_k=\tau_{k,i_k}$ and define
        \begin{align*}
            A_{k+1} &= A_k + a_{k+1}\\
            L_{k+1} &= 2^{i_k-1}L_k\\
            \phi_{k+1}(x) & = \phi_k(x) + a_{k+1}[F(x_{k+1}) + \langle\nabla F(x_{k+1}), x-x_{k+1}\rangle + \Psi(x)].
        \end{align*}
        \ENDFOR
    \end{algorithmic}
\end{algorithm}

The key lemma behind Nesterov~\cite{Nesterov2015UniversalGM}'s analysis of universal methods across these different H\"older-smooth settings is the following unifying condition. In any setting, the gradient of $F$ yields inexact quadratic upper bounds.
\begin{lemma}[Nesterov~\cite{Nesterov2015UniversalGM}, Lemma 2] \label{lem:nesterov-inexactness}
	Suppose $F$ is $(M,v)$-H\"older smooth~\eqref{eq:holder-smooth}. Then for any $\delta>0$ and
	$ L \geq \left[\frac{1-v}{1+v}\cdot\frac{1}{\delta}\right]^{\frac{1-v}{1+v}}M^{\frac{2}{1+v}}$,
	\begin{equation*}
	F(y) \leq F(x)+\langle \nabla F(x), y-x\rangle + \frac{L}{2}\|y-x\|^2 + \frac{\delta}{2}, \quad \forall x,y\in Q \ .
	\end{equation*}
\end{lemma}
This lemma ensures that for large enough $i_k$ the condition~\eqref{eq:descent} is satisfied. In Section~\ref{sec:theory}, we generalize this result to similarly apply to sums of H\"older smooth functions (see Lemma~\ref{lem:inexactness}). As a result, the iterates of $\mathtt{UFGM}$ are well-defined in our more general setting of interest.

On average (amortized), $\mathtt{UFGM}$ uses at most four first-order oracle evaluations $(F(x_k),\nabla F(x_k))$ per iteration.
Stopping criteria and the details of this amortized analysis are given in Nesterov's original development~\cite{Nesterov2015UniversalGM} and we refer any interested reader there.

\section{Motivating Applications and Numerics}\label{sec:applications}
Here we consider {\color{blue} three} applications with heterogeneous sums for the objective function. Simple numerics are conducted {\color{blue} for each} showing the universal fast gradient method converges in much the same fashion as our theory predicts {\color{blue} and contrasting against alternative methods}.

\paragraph{Mixtures of Maximum Likelihoods Models.}
Given observed features $A\in \mathbb{R}^{n\times d}$ and labels $b\in\mathbb{R}^n$, $\ell_p$ regression fits a model $x$ by computing the Maximum Log-likelihood Estimator (MLE) via $\min_{x\in Q} \|Ax-b\|_p^p $
where $Q\subseteq \mathbb{R}^d$ reflects prior knowledge or regularization (for example, imposing nonnegativity $x\geq 0$ or seeking sparsity $\|x\|_1\leq \delta$).
This corresponds to the measurements with Gaussian noise (when $p=2$), Laplacian noise (when $p=1$), and allows for heavy tails (whenever $p<2$). The improved performance of estimators outside of Gaussian settings $p\neq 2$ is well-documented (see~\cite{nyquist1980recent,Money1982,Llorente2013,Efron1986}), although this requires careful analysis to identify a suitable value of $p$.

Suppose data sets $(A_j,b_j)$ from independent sources $j\in\{1\dots J\}$ are aggregated, each with its own, different source of errors with log-likelihoods proportional to $-\|A_jx-b_j\|^{p_j}_{p_j}$ for some $p_j\in [1,2]$. Then the maximum likelihood estimator given all the $J$ data sets is given by
\begin{equation}
    \min_{x\in Q} F(x) = \sum_{j=1}^J \|A_j x -b_j\|^{p_j}_{p_j} \ . \label{eq:mixture-model}
\end{equation}
Observing that each $\|A_j x -b_j\|^{p_j}_{p_j}$ is H\"older smooth with exponent $p_j-1$, we conclude $\mathtt{UFGM}$ can compute the maximum likelihood estimator for aggregated independent data sets in time proportional to that of computing separate maximum likelihood estimators for each data set. {\color{blue} Such problems also arise when the membership information (which samples came from which error distribution) is unknown: the EM algorithm~\cite{Dempster1977,Meng1997} would at each step estimate membership probabilities and then solve a resulting MLE mixture model.}

\begin{figure}\centering
	\begin{subfigure}[b]{0.45\textwidth}
		\includegraphics[width=\textwidth]{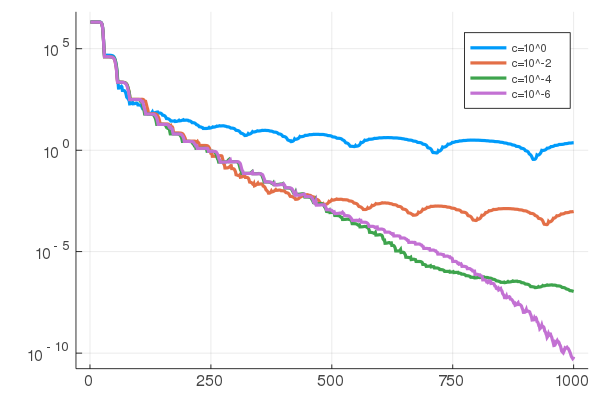}
		\caption{$\mathtt{UFGM}$}
	\end{subfigure}
	\begin{subfigure}[b]{0.45\textwidth}
		\includegraphics[width=\textwidth]{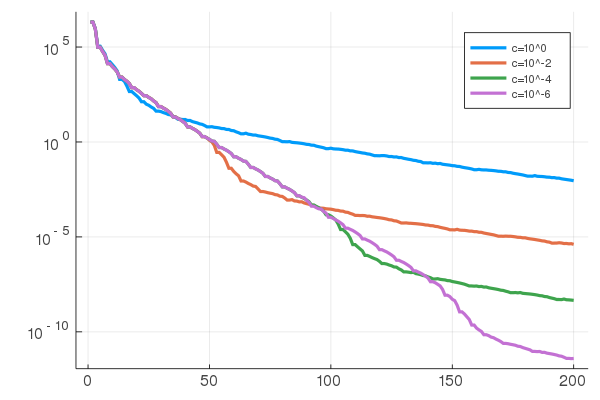}
		\caption{$\mathtt{R\mbox{-}UFGM}$}
	\end{subfigure}
	\caption{Iterations vs Objective Gap from applying $\mathtt{UFGM}$ and $\mathtt{R\mbox{-}UFGM}$ to $F(x)=\frac{1}{2}\|A_1x-b_1\|^2_2 + c\|A_2x-b_2\|_1$  with $x_0=0,L_0=1,\epsilon=10^{-9}$, standard normal $A_i\in\mathbb{R}^{2000\times 1000},x^*\in\mathbb{R}^{1000}$, and $b_i=A_ix^*$ for various values of $c>0$.}\label{fig:p-norm}
\end{figure}
Figure~\ref{fig:p-norm} applies $\mathtt{UFGM}$ and $\mathtt{R\mbox{-}UFGM}$ to $F(x)=\frac{1}{2}\|A_1x-b_1\|^2_2 + c\|A_2x-b_2\|_1$ with varied $c>0$. In all cases, we see fast convergence early on as $\epsilon \approx F(x)-p_*$ is large and so the smooth convergence rate dominates the nonsmooth component $O(1/\sqrt{\epsilon}) >> O(c^2/\epsilon^2)$. For each problem instance, the method suddenly slows down once the nonsmooth rate dominates (around height $\epsilon\approx c^{4/3}$). Noting $F(\cdot)$ here is $\lambda_{min}(A_1^TA_1)$-strongly convex, we see the speedup from restarting predicted by Theorem~\ref{thm:holder-growth-rate}.

\paragraph{Support Vector Machines.}
Consider the unconstrained support vector machine (SVM) training problem given $n$ data points $(x_i,y_i)\in\mathbb{R}^d\times \{\pm 1\}$ and $\lambda>0$
$$ \min_{w\in\mathbb{R}^d} F(w) = \sum_{i=1}^n \max\{0, 1-y_i \cdot x_i^Tw\} + \frac{\lambda}{2}\|w\|_2^2 \ . $$
Note that $F$ is $\lambda$-strongly convex and the sum of an $M$-Lipschitz, nonsmooth function and a function with $\lambda$-Lipschitz gradient. Their sum is neither smooth nor Lipschitz globally. Previous works have overcome this limitation by directly bounding the iterates~\cite{Bach2012} or using a normalized subgradient method~\cite{Grimmer2019}.

Our Theorem~\ref{thm:holder-growth-rate} ensures $\mathtt{R\mbox{-}UFGM}$ finds an $\epsilon$-minimizer within
$$ \frac{32M^2}{\lambda\epsilon} + 4\log\left(\frac{F(w_0)-\inf F}{\epsilon}\right)$$
iterations. The dual averaging methods of~\cite{Chen2012} also attain this fast rate (although their methods need to know several problem constants to be applied). Such fast algorithms are resilient to the choice of $w_0$ as it only appears logarithmically.

{\color{blue}
\paragraph{Orthogonal Projection onto Spectrahedrons.}
Lastly, consider the task of orthogonal projection onto a spectrahedron $\mathcal{S} = \{y\in\mathbb{R}^d \mid C-\mathcal{A}y\succeq 0\}$. Projecting a given $\bar y\in\mathbb{R}^d$ onto $\mathcal{S}$ corresponds to minimizing $\|y-\bar y\|^2_2$ subject to $y\in\mathcal{S}$. Using an exact penalization (under strong duality and strict complementarity assumptions mirroring~\cite[Lemma 6.1]{Ding2021-storage}), for large enough $\alpha>0$, this problem (called SpecProj) amounts to minimizing the sum
$$ \min_{y\in\mathbb{R}^d} F(w) = \frac{1}{2}\|y - \bar y\|_2^2 + \alpha \max\{\lambda_{max}(\mathcal{A}y-C), 0\} \ . $$

Like that of SVM training, this objective is a strongly convex sum of smooth and Lipschitz components; hence, our theory applies.
For both of these problems, we compare the universal performance of $\mathtt{UFGM}$ against various parameterizations of the standard subgradient and regularized dual averaging methods. On three instances from the SVM dataset~\cite{libsvm} and three spectrahedrons $\{ y \mid \mathrm{diag}(y) \preceq L \}$\footnote{Such spectrahedrons occur as the dual feasible region of MAX-CUT SDP relaxations~\cite{Goemans1995}.} with $L$ from the matrix dataset~\cite{Gset}, Figure~\ref{fig:performance} shows $\mathtt{UFGM}$ performs similarly (albeit sometimes slower) to these simpler methods when they are well-tuned. (Some theory for subgradient methods and dual averaging with generic stepsizes in non-Lipschitz settings was recently developed in~\cite{GrimmerLi2023}.)

\begin{figure}\centering
	\begin{subfigure}[b]{0.45\textwidth}
		\includegraphics[width=\textwidth]{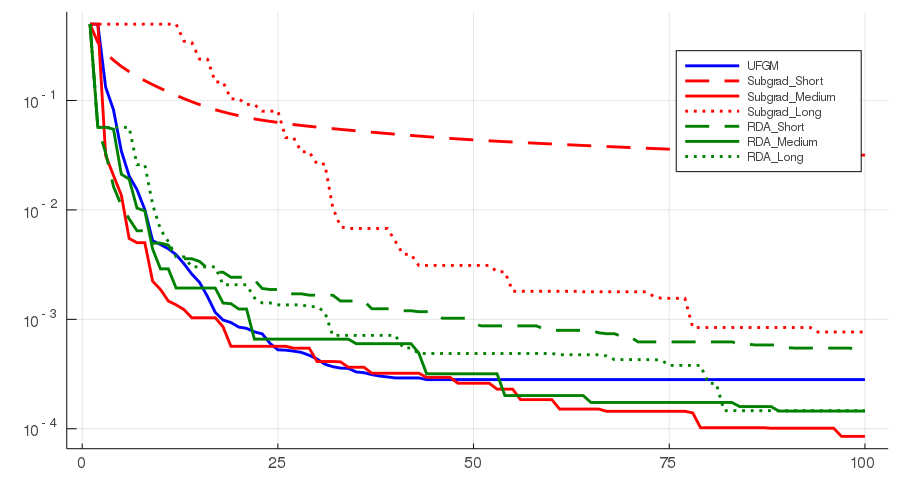}
		\caption{SVM on ``colon-cancer'' data.}
	\end{subfigure}
	\begin{subfigure}[b]{0.45\textwidth}
		\includegraphics[width=\textwidth]{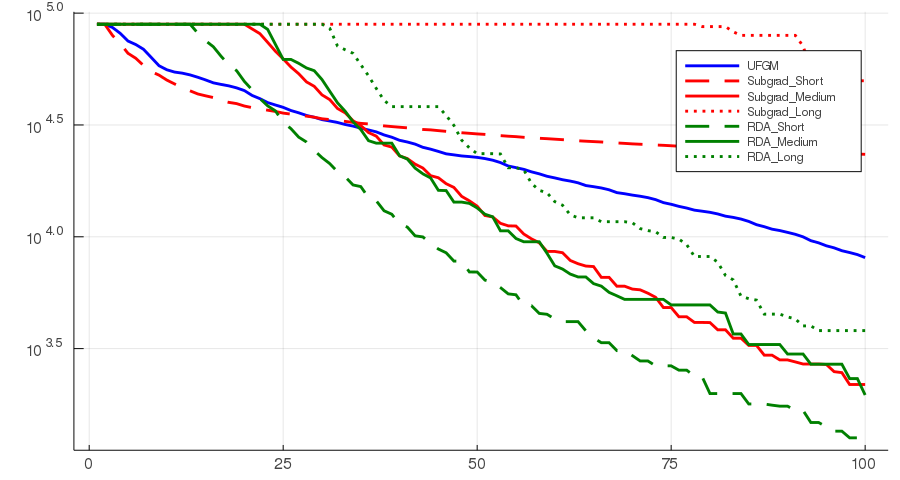}
		\caption{SpecProj on ``G18'' data.}
	\end{subfigure}
	\begin{subfigure}[b]{0.45\textwidth}
		\includegraphics[width=\textwidth]{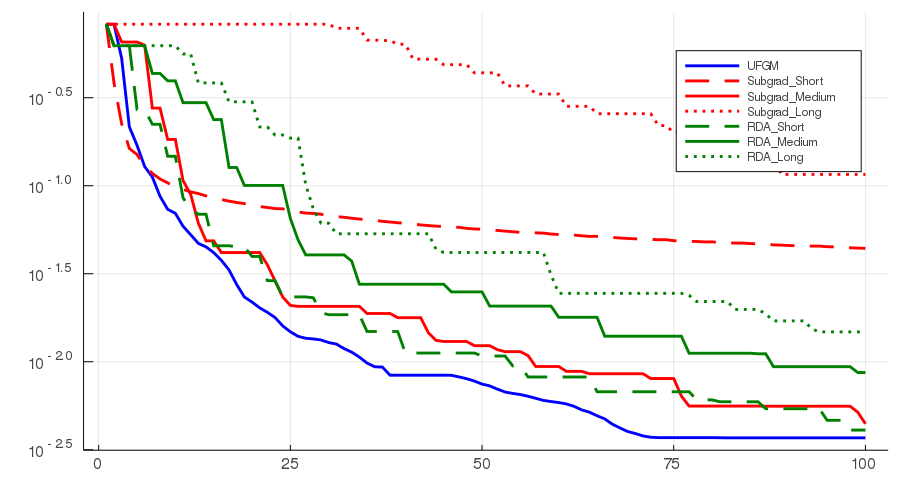}
		\caption{SVM on ``duke'' data.}
	\end{subfigure}
	\begin{subfigure}[b]{0.45\textwidth}
		\includegraphics[width=\textwidth]{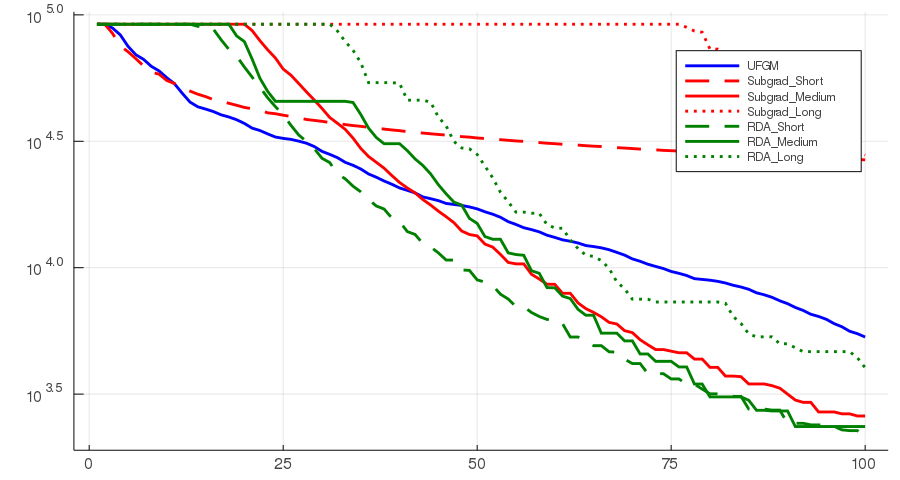}
		\caption{SpecProj on ``G19'' data.}
	\end{subfigure}
 	\begin{subfigure}[b]{0.45\textwidth}
		\includegraphics[width=\textwidth]{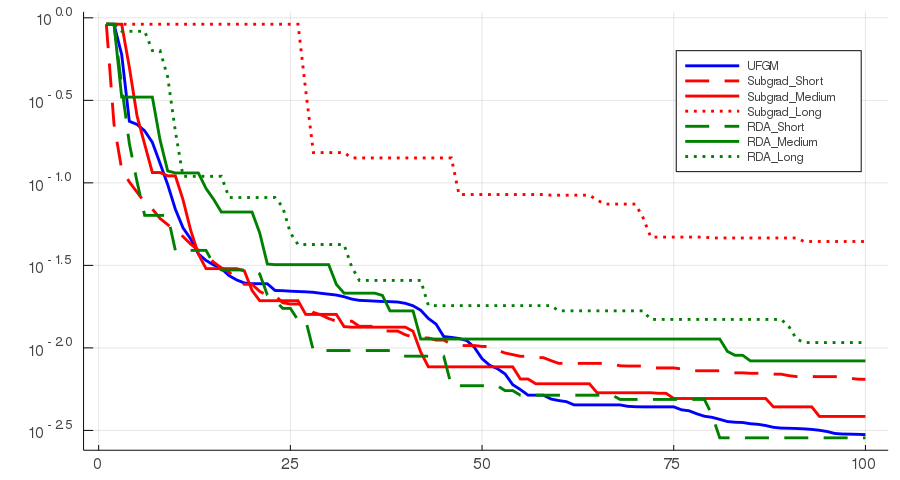}
		\caption{SVM on ``leu'' data.}
	\end{subfigure}
	\begin{subfigure}[b]{0.45\textwidth}
		\includegraphics[width=\textwidth]{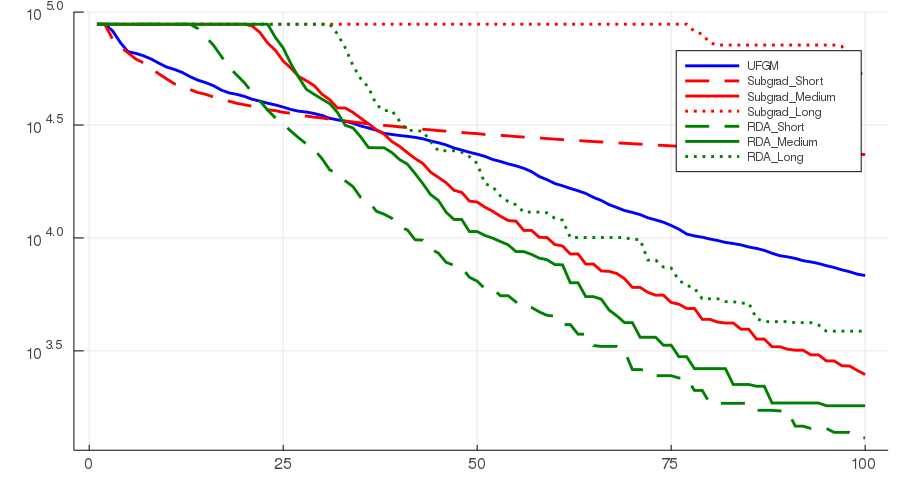}
		\caption{SpecProj on ``G20'' data.}
	\end{subfigure}
    \caption{{\color{blue} Iterations vs Best Objective Gap Seen for $\{$SVM,SpecProj$\}$ instances solved from $x_0=\{0,\bar y\}$ via three methods: $\mathtt{UFGM}$ with $L_0=1$ and $\epsilon=\{10^{-6},10^3\}$, the subgradient method $x_{k+1} = x_k -\alpha_k \nabla F(x_k)$ with short, medium, or long stepsizes $\alpha_k = 0.1/(k+1), 1/(k+1), 10/(k+1)$, and regularized dual averaging~\cite{Xiao2009} with short, medium, or long steps $\lambda_k = 1, k, k^2$ and fixed $\beta_k=\sqrt{k}$. Given both problem types are $1$-strongly convex, the medium steplength settings are theoretically supported by~\cite{GrimmerLi2023}.}}\label{fig:performance}
\end{figure}
}
    \section{Convergence Theory} \label{sec:theory}

Our analysis of $\mathtt{UFGM}$ closely follows the form of Nesterov's original analysis when considering the minimization of a single H\"older smooth function. The primary difference in deriving our more general convergence rates comes from how we estimate the coefficients $A_k$. A simple extension of Lemma~\ref{lem:nesterov-inexactness} is given below, showing subgradient evaluations of the sum $F(x) = \sum_{j\in \mathcal{J}}f_j(x)$ can be viewed as inexact gradient evaluations yielding quadratic upper bounds with constant depending on a combination of the H\"older-smoothness of each $f_j$.
\begin{lemma} \label{lem:inexactness}
	Suppose $F(x) = \sum_{j\in \mathcal{J}}f_j(x)$ with each $f_j$ being $(M_j,v_j)$-H\"older smooth. Then for any $\delta>0$ and
	$ L \geq \sum_{j\in\mathcal{J}} \left[\left[\frac{1-v_j}{1+v_j}\cdot\frac{|\mathcal{J}|}{\delta}\right]^{\frac{1-v_j}{1+v_j}}M_j^{\frac{2}{1+v_j}} \right], $
	\begin{equation*}
		F(y) \leq F(x)+\langle \nabla F(x), y-x\rangle + \frac{L}{2}\|y-x\|^2 + \frac{\delta}{2}, \quad \forall x,y\in Q \ .
	\end{equation*}
\end{lemma}
\begin{proof}
    For each $(M_j,v_j)$-H\"older smooth $f_j$ and
	$ L_j \geq \left[\frac{1-v_j}{1+v_j}\cdot\frac{|\mathcal{J}|}{\delta}\right]^{\frac{1-v_j}{1+v_j}}M_j^{\frac{2}{1+v_j}}, $
	Lemma~\ref{lem:nesterov-inexactness} ensures
	\begin{equation*} 
	    f_j(y) \leq f_j(x)+\langle \nabla f_j(x), y-x\rangle + \frac{L_j}{2}\|y-x\|^2 + \frac{\delta}{2|\mathcal{J}|}, \quad \forall x,y\in Q \ .
	\end{equation*}
	Note that the sum rule allows us to decompose the $\nabla F(x) = \sum_{j\in\mathcal{J}}\nabla f_j(x) $ as a sum of subgradients. Then summing this over all $j$ gives the claim. 
\end{proof}
This lemma ensures the backtracking search in line 4 of Algorithm~\ref{alg:UFGM} always terminates and will further allow us to bound the rate $A_{k}$ grows.
Equip with this inexact oracle result, deriving convergence guarantees for Nesterov's universal fast gradient method follows nearly from the proof of Theorem 3 in~\cite{Nesterov2015UniversalGM}, complicated by a more difficult recurrence relation arising at the end of the argument. We show that the convergence of $\mathtt{UFGM}$ on sums of H\"older smooth functions is controlled by the following recurrence.
\begin{theorem}\label{thm:main-rate-full}
	For any convex $F(x) = \sum_{j\in \mathcal{J}}f_j(x)$ where each $f_j$ is $(M_j,v_j)$-H\"older smooth, all of the iterations of $\mathtt{UFGM}$ are well-defined. Moreover, for any $k\geq 0$,
	\begin{equation}\label{eq:our-bound}
	    A_k \left(F(y_k)+\Psi(y_k)-\frac{\epsilon}{2}\right) \leq \phi_k^* = \min_{x\in Q} \phi_k(x)
	\end{equation}
	where $A_k$ increases monotonically, satisfying the recurrence relation
	\begin{equation}\label{eq:our-recurrence}
	    \sum_{j\in\mathcal{J}}\left[\frac{2|\mathcal{J}|^{\frac{1-v_j}{1+v_j}}M_j^{\frac{2}{1+v_j}}}{\epsilon^{\frac{1-v_j}{1+v_j}}} \frac{(A_{k+1}-A_{k})^{\frac{1+3v_j}{1+v_j}}}{A_{k+1}^{\frac{2v_j}{1+v_j}}}\right] \geq 1\ .
	\end{equation}
\end{theorem}
\begin{proof}
	Our extended Lemma~\ref{lem:inexactness} establishes the iterates of $\mathtt{UFGM}$ are well-defined. Then the guarantee~\eqref{eq:our-bound} exactly follows its derivation in~\cite[Theorem 3]{Nesterov2015UniversalGM}. Furthermore, Lemma~\ref{lem:inexactness} implies that line~3 of the $\mathtt{UFGM}$ has
	$$ 2^{i_k}L_k \leq \sum_{j\in\mathcal{J}} \left[2\left(\frac{|\mathcal{J}|A_{k+1}}{\epsilon a_{k+1}}\right)^{\frac{1-v_j}{1+v_j}} M_j^{\frac{2}{1+v_j}}\right].$$
	Observing that $a_{k+1}^2/A_{k+1} = 1/2^{i_k}L_k$ and $a_{k+1}=A_{k+1}-A_{k}$ yields~\eqref{eq:our-recurrence}.
\end{proof}

From this theorem, the primary difficulty in bounding the convergence of the fast universal method applied to a sum (and thus the primary difficulty in proving Theorems~\ref{thm:main-rate} and~\ref{thm:main-rate-implicit}) is then in solving this recurrence relation. To do so, we prove two bounds on any sequence satisfying the recurrence~\eqref{eq:our-recurrence} in the following two subsections. The first bound proven in Lemma~\ref{lem:recurrence} gives an explicit bound on $k$ in terms of $A_k$, which when combined with~\eqref{eq:our-bound} gives the explicit bound on the accuracy of $\mathtt{UFGM}$ of Theorem~\ref{thm:main-rate}. Our second bound proven in Lemma~\ref{lem:recurrence-implicit} gives an implicit bound for $A_k$ based on the solution of a related nonlinear equation. In turn, this yields the improved (although implicit) guarantee of Theorem~\ref{thm:main-rate-implicit}.
	
\subsection{Proof of Explicit Convergence Guarantee (Theorem~\ref{thm:main-rate})} \label{subsec:explicit}
	From Theorem~\ref{thm:main-rate-full}, every $y_k$ is an $\epsilon$-minimizer of $F$ once $A_k \geq 2\xi(x_0,x^*)/\epsilon$. To ensure $A_k$ reaches this needed size, below we show any recurrence satisfying~\eqref{eq:our-recurrence} has $k$ lower bound a certain summation of powers of $A_k$.
	\begin{lemma}\label{lem:recurrence}
		Suppose a nonnegative, increasing sequence $A_k$ satisfies
		$$ \sum_{j\in\mathcal{J}} \left[\alpha_j \frac{(A_{k+1}-A_k)^{1+q_j}}{A_{k+1}^{q_j}}\right] \geq 1$$
		where $\alpha_j>0$ and $q_j\in [0,1]$ are generic constants for each $j\in\mathcal{J}$. Then for any $k\geq 0$,
		$$ k \leq \sum_{j\in\mathcal{J}} \left[(1+q_j)(\alpha_jA_k)^{\frac{1}{1+q_j}}\right] \ .$$
	\end{lemma}
    \begin{proof}
        Trivially this holds for $k=0$. Inductively suppose the claimed lower bound holds for some $k$.
    	First suppose some single summand $j'\in\mathcal{J}$ has $\alpha_{j'} \frac{(A_{k+1}-A_k)^{1+q_{j'}}}{A_{k+1}^{q_{j'}}} \geq 1$. Then combining this bound with the inductive hypothesis gives our inductive step at $k+1$ as
    	\begin{align*}
    	    k +1
    	    &\leq \sum_{j\in\mathcal{J}} \left[(1+q_j)(\alpha_jA_k)^{\frac{1}{1+q_j}}\right] + \left(\alpha_{j'} \frac{(A_{k+1}-A_k)^{1+q_{j'}}}{A_{k+1}^{q_{j'}}}\right)^{\frac{1}{1+q_{j'}}}\\
    	    &\leq \sum_{j\in\mathcal{J}\setminus\{j'\}} \left[(1+q_j)(\alpha_jA_{k+1})^{\frac{1}{1+q_j}}\right] + (1+q_{j'})(\alpha_{j'}A_k)^{\frac{1}{1+q_{j'}}}+ \alpha_{j'}^{\frac{1}{1+q_{j'}}} \frac{A_{k+1}-A_k}{A_{k+1}^{\frac{q_{j'}}{1+q_{j'}}}}\\
    	    &\leq \sum_{j\in\mathcal{J}\setminus\{j'\}} \left[(1+q_j)(\alpha_jA_{k+1})^{\frac{1}{1+q_j}}\right] + (1+q_{j'})(\alpha_{j'}A_{k+1})^{\frac{1}{1+q_{j'}}}
    	\end{align*}
    	where the first inequality uses the assumption on $j'$ (or rather the $(1+q_{j'})$th root of it), the second uses the monotonicity of $A_k$ on each $j\neq j'$ term, the third uses the concavity of $z^{1/(1+q_{j'})}$ between $A_k$ and $A_{k+1}$.
    	
    	Now suppose instead that every $j\in\mathcal{J}$ has $\alpha_j \frac{(A_{k+1}-A_k)^{1+q_j}}{A_{k+1}^{q_j}} <1$. Then combining our inductive hypothesis with the given recurrence relation  gives our inductive step at $k+1$ as
    	\begin{align*}
    	    k +1 &\leq \sum_{j\in\mathcal{J}} \left[(1+q_j)(\alpha_jA_k)^{\frac{1}{1+q_j}} + \alpha_j \frac{(A_{k+1}-A_k)^{1+q_j}}{A_{k+1}^{q_j}}\right] \\
    	    &\leq \sum_{j\in\mathcal{J}} \left[(1+q_j)(\alpha_jA_k)^{\frac{1}{1+q_j}} + \alpha_{j}^{\frac{1}{1+q_{j}}} \frac{A_{k+1}-A_k}{A_{k+1}^{\frac{q_{j}}{1+q_{j}}}}\right] \\
    	    &\leq \sum_{j\in\mathcal{J}} \left[(1+q_j)(\alpha_jA_{k+1})^{\frac{1}{1+q_j}}\right]
    	\end{align*}
    	where the second inequality uses our assumed bound on every $j$ (and that $z^{1/(1+q_{j})} \geq z$ for all $j\geq0$ and $z\in[0,1]$) and the third uses the concavity of $z^{1/(1+q_{j})}$ between $A_k$ and $A_{k+1}$.
    \end{proof}
	Lemma~\ref{lem:recurrence} suffices to complete our proof of Theorem~\ref{thm:main-rate} as it follows that
    $$ k \leq \sum_{j\in\mathcal{J}}\left[\frac{1+3v_j}{1+v_j}\left(\frac{2|\mathcal{J}|^{\frac{1-v_j}{1+v_j}}M_j^{\frac{2}{1+v_j}}}{\epsilon^{\frac{1-v_j}{1+v_j}}}A_{k}\right)^{\frac{1+v_j}{1+3v_j}} \right]\ . $$ 
    Plugging in $A_k\geq 2\xi(x_0,x^*)/\epsilon$ then gives the result.

\subsection{Proof of Improved Implicit Convergence Guarantee (Theorem~\ref{thm:main-rate-implicit})} \label{subsec:implicit}
Here we improve on the convergence guarantee of Theorem~\ref{thm:main-rate} by providing a tighter analysis of the recurrence relation~\eqref{eq:our-recurrence} than Lemma~\ref{lem:recurrence} provides in the following Lemma~\ref{lem:recurrence-implicit}. By applying this lemma in the place of Lemma~\ref{lem:recurrence}, Theorem~\ref{thm:main-rate-implicit} immediately follows.

\begin{lemma}\label{lem:recurrence-implicit}
	Suppose a nonnegative, increasing sequence $A_k$ satisfies
	$$ \sum_{j\in\mathcal{J}} \left[\alpha_j \frac{(A_{k+1}-A_k)^{1+q_j}}{A_{k+1}^{q_j}}\right] \geq 1$$
	where $\alpha_j>0$ and $q_j\in [0,1]$ are generic constants for each $j\in\mathcal{J}$. Then for any $\Delta>0$, $A_k \geq \Delta$ for all $k \geq 5C$ where $C$ is the unique positive root of the equation
	$$ \sum_{j\in \mathcal{J}} \alpha_j\Delta C^{-(1+q_j)} -1 = 0 \ . $$
\end{lemma}
\begin{proof}
    First, observe that since $A_0\geq 0$, it follows that $A_1 \geq 1/\sum \alpha_j$. Hence without loss of generality, we can assume $\Delta > 1/\sum \alpha_j$ (as the result is immediate otherwise since $A_k\geq A_1>\Delta$).
    
    For any $\delta \geq 1/\sum \alpha_j$, let $C(\delta)>1$ denote the unique positive root of
    $$\sum_{j\in \mathcal{J}} \alpha_j \delta C^{-(1+q_j)} - 1 = 0 \ .$$
    Uniqueness of the positive solution $C(\delta)$ follows from the fact that the function $C\mapsto \sum \alpha_j \delta C^{-1-q_j}-1$ is strictly decreasing for $C>0$ and approaches $-1$ as $C\rightarrow \infty$. Existence of a solution $C(\delta)>1$ follows as this function equals $\sum \alpha_j \delta -1 > 0$ at $C=1$. As a final useful property of $C(\delta)$, observe that for any $\lambda\leq 1$,
    \begin{equation}\label{eq:C-property}
        \lambda C(\delta) \leq C(\lambda \delta) \leq \sqrt{\lambda}C(\delta) \ .
    \end{equation}
    
    Let $N:=\lceil\log_2(\Delta/A_1)\rceil$, which ensures $2^{N-1} A_1\leq \Delta \leq 2^N A_1$. Then we prove the lemma by showing that for any $n \in \{0, \dots, N-1\}$, at most
    \begin{equation} \label{eq:interval-bound}
        \log(2)C(2^{n+1}A_1)
    \end{equation}
    many different values of $k$ have $A_k$ in the interval $[2^n, 2^{n+1}]/\sum\alpha_j$. Summing this up from $n=0$ to $N-1$ gives the claim the number of steps before $A_k \geq 2^NA_1 \geq \Delta$ is at most
    \begin{align*}
        \sum_{n=0}^{N-1} \log(2)C(2^{n+1}A_1) &\leq \sum_{n=0}^{N-1} \log(2)C(2^{N}A_1)\sqrt{2}^{n+1-N}\\
        &\leq  \frac{\sqrt{2}\log(2)C(2^NA_1)}{(\sqrt{2}-1)} < 5C(\Delta)
    \end{align*}
    where the first inequality uses~\eqref{eq:C-property}, the second upper bounds this sum by a geometric series, and the third again uses~\eqref{eq:C-property} (and then rounds the coefficient $2\sqrt{2}\log(2)/(\sqrt{2}-1)$ up to $5$ for simplicity).
    
    Now we complete the proof by showing the claimed bound~\eqref{eq:interval-bound} for any $n\geq 1$. Note for any $A_{k+1} \leq 2^{n+1}A_1$, the given recurrence relation implies that
	$$ \sum_{j\in\mathcal{J}} \left[\alpha_j2^{n+1}A_1\left(\frac{A_{k+1}-A_k}{A_{k+1}}\right)^{1+q_j}\right] \geq 1 \ .$$
	Hence $(A_{k+1}-A_k)/A_{k+1}$ must be at least $C(2^{n+1}A_1)^{-1}$, or equivalently,
	$$ A_{k+1} \geq \frac{C(2^{n+1}A_1)}{C(2^{n+1}A_1)-1}A_k \ . $$
	Thus the value of $A_{k}$ grows geometrically within the interval $[2^n, 2^{n+1}]/\sum\alpha_j$. As a result, the number of different $k$ with $A_k$ in this interval is at most
	\begin{align*}
	    \log_{C(2^{n+1}A_1)/(C(2^{n+1}A_1)-1)}(2) &= \frac{\log(2)}{\log(C(2^{n+1}A_1)/(C(2^{n+1}A_1)-1))}\\
	    &\leq \log(2)C(2^{n+1}A_1))\ .
	\end{align*}
	where the inequality utilizes that $x \geq 1/\log(x/(x-1))$ for all $x\geq 1$.
\end{proof}

\subsection{Improved Theory Under H\"older Growth Bounds}\label{subsec:strong-convexity}
For many first-order methods in a range of different settings, faster convergence guarantees are well-known under strong convexity of $F+\Psi$ or more generally under some growth/error bound or KL condition. The most common such setting is $\mu$-strongly convex optimization, which possesses $(\mu,2)$-H\"older growth. Recall for $(M,v)$-H\"older smooth optimization satisfying $(\mu,p)$-H\"older growth~\eqref{eq:holder-growth}, the optimal\footnote{The original source for such a lower bound is difficult to identify in the literature, leaving this optimality somewhat as folklore. The theory of~\cite{Grimmer2021rateLifting} allows these bounds to be concluded from the classic lower bounds without growth.} rate is
$O\left(\frac{M^{\frac{2}{1+3v}}}{\mu^{\frac{2(1+v)}{p(1+3v)}}\epsilon^{\frac{2(p-1-v)}{p(1+3v)}}}\right)$
if $v<p-1$ and $O((M/\mu)^{\frac{2}{1+3v}}\log(1/\epsilon))$ if $v=p-1$.

Our Theorem~\ref{thm:holder-growth-rate} generalizes this fact to sums for the following restarted variant dubbed $\mathtt{R\mbox{-}UFGM}$, in Algorithm~\ref{alg:RUFGM}. We will assume that the optimal value of~\eqref{eq:our-problem} is known for ease of development. The works~\cite{Roulet2020,RenegarGrimmer2018} have provided more sophisticated restarting approaches, avoiding such assumptions, often at the cost of a logarithmic term in the oracle complexity. These schemes may give a wholly parameter-free approach but is beyond the scope of this work.
\begin{algorithm}
    \caption{Restarted Universal Fast Gradient Method ($\mathtt{R\mbox{-}UFGM}$)}\label{alg:RUFGM}
    \begin{algorithmic}[1]
        \STATE \textbf{Initialization}: Choose $z_0\in Q$, $L_0>0$. Define $\epsilon_0 = (F(z_0)+\Psi(z_0)-p_*)/2$.
        \FOR{$n = 0,1,2,\dots$}
        \STATE Let $y_k^{(n)}$ denote the iterates of $\mathtt{UFGM}(z_n, \epsilon_n, L_0)$ run until $y^{(n)}_{k_n}$ is an $\epsilon_n$-minimizer.
        \STATE Set $z_{n+1} = y_{k_n,n}$ and $\epsilon_{n+1}=\epsilon_n/2$.
        \ENDFOR
    \end{algorithmic}
\end{algorithm}

Theorem~\ref{thm:holder-growth-rate} has the convergence rate with respect to each summand with $v_j+1 \leq p$ improve to the optimal convergence rate~\eqref{eq:optimal-growth-rate} for $(M_j,v_j)$-H\"older smooth, $(\mu, p)$-H\"older growth optimization. Each summand in our convergence rate with $v_j +1 > p$ decreases superlinearly. These rapidly decaying terms cannot be compared to an isolated setting as no function can have $(M,v)$-H\"older smoothness and $(\mu,p)$-H\"older growth with $v+1 > p$.

\subsubsection{Proof of Improved H\"older Growth Convergence Guarantee (Theorem~\ref{thm:holder-growth-rate})}
    Observe that the stopping criteria for each run of $\mathtt{UFGM}$ ensures that each initialization $z_n$ has
    $F(z_n)+\Psi(z_n)-p_* \leq 2^{-n}(F(z_0)+\Psi(z_0)-p_*) = 2\epsilon_n.$
    It follows that after $N$ restarts, an $\tilde\epsilon$-minimizer has been found. First, we bound the number of iterations needed to reach the stopping criteria for each iteration $n$ of $\mathtt{R\mbox{-}UFGM}$.
    Namely, Theorem~\ref{thm:main-rate} and H\"older growth imply
    \begin{align*}
        k_n -1 & \leq \sum_{j\in\mathcal{J}}\left[\left(\frac{1+3v_j}{1+v_j}2^{\frac{1+v_j}{1+3v_j}}|\mathcal{J}|^{\frac{1-v_j}{1+3v_j}}\right)\left(\frac{M_j}{\epsilon_n}\right)^{\frac{2}{1+3v_j}}\xi(z_n,x^*)^{\frac{1+v_j}{1+3v_j}} \right]\\
        &\leq \sum_{j\in\mathcal{J}}\left[\left(\frac{1+3v_j}{1+v_j}2^{\frac{1+v_j}{1+3v_j}}|\mathcal{J}|^{\frac{1-v_j}{1+3v_j}}\right)\left(\frac{M_j}{\epsilon_n}\right)^{\frac{2}{1+3v_j}}\left(\frac{\epsilon_n}{\mu}\right)^{\frac{2(1+v_j)}{p(1+3v_j)}} \right]\\
        & = \sum_{j\in\mathcal{J}}\left[\left(\frac{1+3v_j}{1+v_j}2^{\frac{1+v_j}{1+3v_j}}|\mathcal{J}|^{\frac{1-v_j}{1+3v_j}}\right)\frac{M_j^{\frac{2}{1+3v_j}}}{\mu^{\frac{2(1+v_j)}{p(1+3v_j)}}\epsilon_n^{\frac{2(p-1-v_j)}{p(1+3v_j)}}}\right] \ .
    \end{align*}
    Then bounding $\epsilon_n \geq 2^{N-n-1}\tilde\epsilon$, the total iterations used by $\mathtt{R\mbox{-}UFGM}$ is
    \begin{align}
        \sum_{n=0}^Nk_n &\leq \sum_{n=0}^N\sum_{j\in\mathcal{J}}\left[\left(\frac{1+3v_j}{1+v_j}2^{\frac{1+v_j}{1+3v_j}}|\mathcal{J}|^{\frac{1-v_j}{1+3v_j}}\right)\frac{M_j^{\frac{2}{1+3v_j}}}{\mu^{\frac{2(1+v_j)}{p(1+3v_j)}}(2^{N-n-1}\tilde\epsilon)^{\frac{2(p-1-v_j)}{p(1+3v_j)}}}\right] +N\nonumber\\
        &=\sum_{j\in\mathcal{J}}\left[c_j'''\frac{M_j^{\frac{2}{1+3v_j}}}{\mu^{\frac{2(1+v_j)}{p(1+3v_j)}}\tilde \epsilon^{\frac{2(p-1-v_j)}{p(1+3v_j)}}}\sum_{n=0}^N\frac{1}{(2^{N-n})^{\frac{2(p-1-v_j)}{p(1+3v_j)}}}\right] +N\nonumber\\
        &\leq\sum_{j\in\mathcal{J}}\left[c_j'''\frac{M_j^{\frac{2}{1+3v_j}}}{\mu^{\frac{2(1+v_j)}{p(1+3v_j)}}\tilde \epsilon^{\frac{2(p-1-v_j)}{p(1+3v_j)}}}\min\left\{\sum_{r=0}^\infty\left(\frac{1}{2^{\frac{2(p-1-v_j)}{p(1+3v_j)}}}\right)^r, \frac{N}{2^{\frac{2(p-1-v_j)}{p(1+3v_j)}}}\right\}\right] +N \ .\nonumber
    \end{align}
    Bounding this geometric sum gives the claimed result.

    {\small
    \bibliographystyle{unsrt}
    \bibliography{bibliography}
    }

\end{document}